\documentclass[11pt]{amsart}%
\usepackage{amssymb}
\usepackage{amsmath}
\usepackage{graphicx}
\usepackage{amsfonts}%
\newtheorem{theorem}{Theorem}[section]
\theoremstyle{plain}

\newtheorem{lemma}{Lemma}[section]

\newtheorem{proposition}{Proposition}[section]

\numberwithin{equation}{section}
\begin{document}
\title[A Liouville Type Theorem]{A Liouville type theorem for special Lagrangian Equations with constraints}
\author{Micah Warren}
\author{Yu YUAN}
\address{Department of Mathematics, Box 354350\\
University of Washington\\
Seattle, WA 98195}
\email{mwarren@math.washington.edu, yuan@math.washington.edu}
\thanks{Y.Y. is partially supported by an NSF grant.}
\date{\today}

\begin{abstract}
We derive a Liouville type result for special Lagrangian equations with
certain ``convexity'' and restricted linear growth assumptions on the solutions.

\end{abstract}
\maketitle

\section{\bigskip Introduction}

In this note, we show the following

\begin{theorem}
Let $u$ be a smooth solution to the special Lagrangian equation
\begin{equation}
\sum_{i=1}^{n}\arctan\lambda_{i}=c\ \ \ \text{on \ \ }\mathbb{R}^{n},
\label{EsLag}%
\end{equation}
where $\lambda_{i}s$ are the eigenvalues of the Hessian $D^{2}u\left(
x\right)  .$ Suppose that
\begin{equation}
3+(1-\varepsilon)\lambda_{i}^{2}\left(  x\right)  +2\lambda_{i}\left(
x\right)  \lambda_{j}\left(  x\right)  \geq0\quad\label{Cond}%
\end{equation}
for all $i,\ j,\ x$ and any small fixed $\varepsilon>0;$ and the gradient
$\nabla u(x)$ satisfies
\begin{equation}
|\nabla u(x)|\leq\delta\left(  n\right)  |x| \label{Unpleasant}%
\end{equation}
for large $\left|  x\right|  $ and any fixed $\delta\left(  n\right)
<1/\sqrt{n-1}.$ Then $u$ must be a quadratic polynomial.
\end{theorem}

The special Lagrangian equation (\ref{EsLag}) arises in the calibrated
geometry [HL]. A Lagrangian graph $M=\left(  x,\nabla u\left(  x\right)
\right)  \subset\mathbb{C}^{n}=\mathbb{R}^{n}\times\mathbb{R}^{n}$ is called
\emph{special }when the calibrating $n$-form
\[
\Omega_{c}=\operatorname{Re}(e^{-\sqrt{-1}\,c}dz_{1}\wedge dz_{2}%
\wedge...\,\wedge dz_{n})
\]
is equal to the induced volume form along $M;$ equivalently, $u$ satisfies
(\ref{EsLag}). The equation (1.1)\ holds if and only if the gradient graph
$\left(  x,\nabla u\left(  x\right)  \right)  \subset\mathbb{C}^{n}$ is a
(volume minimizing) minimal surface in $\mathbb{R}^{n}\times\mathbb{R}^{n}$
[HL, Theorem 2.3, Proposition 2.17].

By Fu's classification result [F], any global solution to (\ref{EsLag}) on
$\mathbb{R}^{2}$ is either quadratic or harmonic; a harmonic function with any
linear growth condition on the gradient is certainly quadratic; see also [Y3]
for a uniqueness result for the global solutions to (\ref{EsLag}) with
$\left\vert c\right\vert >\left(  n-2\right)  \frac{\pi}{2}$. In the case
$n=3,$ other Liouville-Bernstein type results hold true for (\ref{EsLag})
under the following conditions respectively: $\lambda_{i}\geq-K$ [Y2];
$\lambda_{i}\lambda_{j}\geq-K$ [Y4]; or\ $c=\pi$ and the solution is strictly
convex with quadratic growth [BCGJ]. While boundedness of the Hessian alone is
sufficient in dimension three, certain boundedness and convexity are both
needed for Liouville-Bernstein type results to be valid for (\ref{EsLag}) in
the general dimension ($n\geq4$). The results hold with the assumptions that
$c=k\pi$ and the solution is convex with linear growth [B]; with the almost
convex assumption $\lambda_{i}\geq-\varepsilon\left(  n\right)  $ [Y2]; with
the semi-convex assumption $\lambda_{i}\geq-\frac{1}{\sqrt{3}}+\gamma$
everywhere, or with the (\textquotedblleft equivalent\textquotedblright)
assumption $\left\vert \lambda_{i}\right\vert \leq\sqrt{3}-\gamma^{\prime}$
everywhere [Y4]; or with the assumption $\lambda_{i}\lambda_{j}\geq
-1-\varepsilon\left(  n\right)  \ $[Y4]. (It is straightforward that any
convex solution with a bounded Hessian to (\ref{EsLag}) is a quadratic
polynomial, by the well-known $C^{\alpha}$ Hessian estimate of Krylov-Evans
for now the convex elliptic equation (\ref{EsLag}); see also [X, p. 217--218]
for a different approach via the iteration argument of [HJW].) A
Liouville-Bernstein type result with the assumption $\left\vert \lambda
_{i}\right\vert \leq K$ and $\lambda_{i}\lambda_{j}\geq const>-\frac{3}{2}$
was stated in [TW].

The more general ``convexity''\ condition (\ref{Cond}) does not alone lead to
any Hessian bound for the solutions to (\ref{EsLag}), but does guarantees that
the volume element $V,$ which is a geometric combination of the eigenvalues,
is subharmonic. Better yet, the Laplacian of $V$ bounds its gradient; see
Lemma 2.1, which is a key piece in our proof of Lemma 2.2 on our Hessian estimates.

In fact, this paper grows out of our attempts towards deriving a Hessian
estimate in terms of the gradient, for solutions to the special Lagrangian
equation (\ref{EsLag}). The unpleasant technical assumption $\delta\left(
n\right)  <1/\sqrt{n-1}$ in (\ref{Unpleasant}) reflects the limitation of our
current arguments; the assumption is necessary for us to push the
Bernstein-Pogorelov-Korevaar technique to obtain a Hessian estimate for
special Lagrangian equations; see Lemma 2.2.

Once a Hessian bound for solutions to (\ref{EsLag}) is available, the
\textquotedblleft standard\textquotedblright\ blow-down process from the
geometric measure theory will show that the global solution is a quadratic
polynomial, provided certain convexity conditions like (\ref{Cond}) or others
are available in the \emph{whole} process (for $n\geq4).$ (Unlike [JX], we
could not generalize the iteration argument in [HJW] to get a Liouville type
result for now the larger image set (\ref{Cond}) of the corresponding harmonic
Gauss map to the Lagrangian Grassmanian.) The simple constraints $\left\vert
\lambda_{i}\right\vert \leq K$ like $\left\vert \lambda_{i}\right\vert \leq1$
or $\left\vert \lambda_{i}\right\vert \leq\sqrt{3}-\gamma$ are easily shown to
be available in the blow-down process. An \emph{extra} effort is needed to
justify that the nonlinear  constraints (\ref{Cond}) or others like
$\lambda_{i}\lambda_{j}\geq const$ are preserved under the $C^{1,\alpha}$
convergence of the scaling process $u_{k}\left(  x\right)  =u\left(
k^{2}x\right)  /k^{2}.$ Taking advantage of the single elliptic equation
(\ref{EsLag}), we apply the $W^{2,\delta}$ estimates for solutions in terms of
the supreme norm of the solution to extract a $W^{2,\delta}$ sub-convergent
sequence, as in [Y1]. Then we extract another subsequence with the Hessians
converging almost everywhere. This justifies that the constraints (\ref{Cond})
are preserved in the above blow-down process. \ Another route of the
justification is through Allard's regularity result (cf. [S, Section 36]).

Actually, Theorem 1.1 holds true for\ $n=3$ without any growth condition like
(\ref{Unpleasant})$.$ The condition (\ref{Cond}) implies $\lambda_{i}%
\lambda_{j}\geq-K,$ so as in [Y4] we can find a bound on the Hessian (possibly
for a new potential), and then draw the conclusion. Note that the boundedness
on the Hessian alone for $n=3$ is enough for one to run the blow-down process
to obtain a Liouville type result; see [F-C, Theorem 5.4] of Fischer-Colbrie.
In general dimension $n\geq4,$ we derive yet another Liouville-Bernstein type
result for the solutions to (\ref{EsLag}) with the bounded Hessian satisfying
weaker constraints (\ref{W-Cond}); see Theorem 3.1 in the appendix. One
consequence of Theorem 3.1 coupled with the De Giorgi-Allard $\varepsilon
$-regularity theory is an improvement of the above mentioned
Liouville-Bernstein type result in [Y4], namely, any global solution to
(\ref{EsLag}) with $\lambda_{i}\geq-\frac{1}{\sqrt{3}}-\varepsilon\left(
n\right)  $ everywhere or $\left|  \lambda_{i}\right|  \leq\sqrt
{3}+\varepsilon^{\prime}\left(  n\right)  $ everywhere is a quadratic
polynomial (for $n\geq4)$. The argument is identical to the one in [Y2] with
Proposition 2.1 there replaced by Proposition 3.1 here.

The desired Hessian estimate for special Lagrangian equations in the two
dimensional case follows from the gradient estimates in terms of the heights
of the two dimensional minimal graphs with any codimension by Gregori [G],
where some Jacobian estimates of Heinz were employed. For higher dimensional
and codimensional minimal graphs with the assumption that the product of any
two slopes is between $-1$ and $1$, the gradient estimates were obtained in
[W], using an integral method developed for codimension one minimal graphs.
The gradient estimate for codimension one minimal graphs is by now a classical result.

The general Hessian estimate for special Lagrangian equations is still a
puzzling issue to us.

\textbf{Notation.} $\partial_{i}=\frac{\partial}{\partial x_{i}}%
,\ \partial_{ij}=\frac{\partial^{2}}{\partial x_{i}\partial x_{j}}%
,\ u_{i}=\partial_{i}u,\ u_{ji}=\partial_{ij}u,$ etc.

\section{Proof Of Theorem 1.1}

Taking the gradient of both sides of the special Lagrangian equation
(\ref{EsLag}), we have
\begin{equation}
\sum_{i,j}^{n}g^{ij}\partial_{ij}\left(  x,\nabla u\left(  x\right)  \right)
=0, \label{Emin}%
\end{equation}
where $\left(  g^{ij}\right)  $ is the inverse of the induced metric
$g=\left(  g_{ij}\right)  =I+D^{2}uD^{2}u$ on the surface $\left(  x,\nabla
u\left(  x\right)  \right)  \subset\mathbb{R}^{n}\times\mathbb{R}^{n}.$ Simple
geometric manipulation of (\ref{Emin}) yields the usual form of the minimal
surface equation
\[
\bigtriangleup_{g}\left(  x,\nabla u\left(  x\right)  \right)  =0,
\]
where the Laplace-Beltrami operator of the metric $g$ is given by
\[
\bigtriangleup_{g}=\frac{1}{\sqrt{\det g}}\sum_{i,j}^{n}\partial_{i}\left(
\sqrt{\det g}g^{ij}\partial_{j}\right)  .
\]
Because we are using harmonic coordinates $\bigtriangleup_{g}x=0,$ we see that
$\bigtriangleup_{g}$ also equals the linearized operator of the special
Lagrangian equation (\ref{EsLag}) at $u,$%
\[
\bigtriangleup_{g}=\sum_{i,j}^{n}g^{ij}\partial_{ij}.
\]
The gradient and inner product with respect to the metric $g$ are
\begin{align*}
\nabla_{g}v  &  =\left(  \sum_{k=1}^{n}g^{1k}v_{k},\cdots,\sum_{k=1}^{n}%
g^{nk}v_{k}\right) \\
\left\langle \nabla_{g}v,\nabla_{g}w\right\rangle _{g}  &  =\sum_{i,j=1}%
^{n}g^{ij}v_{i}w_{j},\ \ \text{in particular \ }\left|  \nabla_{g}v\right|
_{g}^{2}=\left\langle \nabla_{g}v,\nabla_{g}v\right\rangle _{g}.
\end{align*}

We begin by demonstrating a Jacobi inequality for the volume element
\[
V=\sqrt{\det g}=\prod_{i=1}^{n}(1+\lambda_{i}^{2})^{\frac{1}{2}}.
\]

\begin{lemma}
Suppose that $u$ is a smooth solution to (\ref{EsLag}) satisfying
(\ref{Cond}). \ Then%

\[
\bigtriangleup_{g}\ln V\;\geq\frac{\varepsilon}{n}|\nabla_{g}\ln V|_{g}^{2}%
\]
or equivalently
\begin{equation}
\bigtriangleup_{g}V^{\frac{\varepsilon}{n}}\geq2\frac{|\nabla_{g}%
V^{\frac{\varepsilon}{n}}|_{g}^{2}}{V^{\frac{\varepsilon}{n}}}. \label{Jac1}%
\end{equation}

\end{lemma}

\begin{proof}
\ By differentiating the minimal surface equation (\ref{Emin}) again and
performing some long and tedious computation, one gets the standard formula
for $\bigtriangleup_{g}\ln V;$ see for example [Y2, Lemma 2.1]. (The general
formula for minimal submanifolds of any dimension or codimension originates in
Simons [Ss, p. 90].)\ At any fixed point, we assume that $D^{2}u$ is
diagonalized, then
\[
\bigtriangleup_{g}\ln V=\sum_{i,j,k=1}^{n}(1+\lambda_{i}\lambda_{j}%
)h_{ijk}^{2},
\]
where $h_{ijk}=\sqrt{g^{ii}}\sqrt{g^{jj}}\sqrt{g^{kk}}u_{ijk}.$ Gathering all
terms containing $h_{ijj}^{2}=h_{jij}^{2}=h_{jji}^{2}$ for a fixed $i$, we
have
\begin{align*}
&  (1+\lambda_{i}^{2})h_{iii}^{2}+\sum_{j\neq i}(1+\lambda_{j}^{2})h_{jji}%
^{2}+\sum_{j\neq i}(1+\lambda_{i}\lambda_{j})h_{ijj}^{2}+\sum_{j\neq
i}(1+\lambda_{j}\lambda_{i})h_{jij}^{2}\\
&  =(1+\lambda_{i}^{2})h_{iii}^{2}+\sum_{j\neq i}(3+\lambda_{j}^{2}%
+2\lambda_{i}\lambda_{j})h_{jji}^{2}.
\end{align*}
Thus
\begin{align}
\bigtriangleup_{g}\ln V\;  &  =\sum_{i=1}^{n}\left[  (1+\lambda_{i}%
^{2})h_{iii}^{2}+\sum_{j\neq i}(3+\lambda_{j}^{2}+2\lambda_{i}\lambda
_{j})h_{jji}^{2}\right] \label{Simons}\\
&  +2\sum_{i<j<k}(3+\lambda_{i}\lambda_{j}+\lambda_{j}\lambda_{k}+\lambda
_{k}\lambda_{i})h_{ijk}^{2}.\nonumber
\end{align}
Condition (\ref{Cond}) gives that
\[
\underline{3}+\left(  1-\varepsilon\right)  \lambda_{i}^{2}+\lambda_{i}%
\lambda_{j}+\underline{\lambda_{i}\lambda_{j}+\lambda_{k}(\lambda_{j}%
+\lambda_{i})}-\lambda_{k}(\lambda_{j}+\lambda_{i})\geq0
\]
that is
\[
S_{ijk}=\underline{3+\lambda_{i}\lambda_{j}+\lambda_{j}\lambda_{k}+\lambda
_{k}\lambda_{i}}\geq(\lambda_{k}-\lambda_{i})(\lambda_{i}+\lambda
_{j})+\varepsilon\lambda_{i}^{2}.
\]
Switching $\lambda_{i}$ and $\lambda_{j,}$ we also have
\[
S_{ijk}=S_{jik}\geq(\lambda_{k}-\lambda_{j})(\lambda_{j}+\lambda
_{i})+\varepsilon\lambda_{j}^{2}.
\]
By symmetry of \ $S_{ijk},$ we may assume
\begin{equation}
\lambda_{i}\geq\lambda_{k}\geq\lambda_{j}, \label{Squeez}%
\end{equation}
then either $(\lambda_{k}-\lambda_{i})(\lambda_{i}+\lambda_{j})$ or
$(\lambda_{k}-\lambda_{j})(\lambda_{j}+\lambda_{i})$ has to be non-negative,
thus
\begin{equation}
S_{ijk}\geq\varepsilon\min\left\{  \lambda_{i}^{2},\lambda_{j}^{2}\right\}  .
\label{Lowerbound}%
\end{equation}
We conclude that
\begin{equation}
\bigtriangleup_{g}\ln V\geq\sum_{i}^{n}\left[  (1+\lambda_{i}^{2})h_{iii}%
^{2}+\sum_{j\neq i}(3+\lambda_{j}^{2}+2\lambda_{i}\lambda_{j})h_{jji}%
^{2}\right]  . \label{Diagonal}%
\end{equation}
To bound the gradient, we compute, (still at the same fixed point with
$D^{2}u$ diagonalized,)
\[
\partial_{i}\ln V=\sum_{j=1}^{n}g^{jj}\lambda_{j}u_{jji},
\]
then
\begin{align}
|\nabla_{g}\ln V|_{g}^{2}  &  =\sum_{i=1}^{n}g^{ii}\left(  \sum_{j=1}%
^{n}g^{jj}\lambda_{j}u_{jji}\right)  ^{2}\nonumber\\
&  =\sum_{i=1}^{n}\left(  \sum_{j=1}^{n}\lambda_{j}h_{jji}\right)  ^{2}\leq
n\sum_{i,j=1}^{n}\lambda_{j}{}^{2}h_{jji}^{2}. \label{Gradient}%
\end{align}
Combining (\ref{Cond}) with (\ref{Diagonal}) and (\ref{Gradient}) we have
\begin{align}
&  \bigtriangleup_{g}\ln V-\frac{\varepsilon}{n}|\nabla_{g}\ln V|^{2}%
\nonumber\\
&  \geq\sum_{i=1}^{n}(\left[  1+\left(  1-\varepsilon\right)  \lambda_{i}%
^{2}\right]  h_{iii}^{2}+\sum_{j\neq i}(\left[  3+(1-\varepsilon)\lambda
_{j}^{2}+2\lambda_{i}\lambda_{j}\right]  h_{jji}^{2}\geq0. \label{Jac-Diag}%
\end{align}
The proof of Lemma 2.1 is complete.\textbf{\textit{\ }}
\end{proof}

\begin{lemma}
Suppose that $u$ is a smooth solution to (\ref{EsLag}) on $B_{1}(0)$
satisfying condition (\ref{Cond}) and
\[
|\nabla u|\leq\delta<\frac{1}{\sqrt{n-1}}.
\]
Then
\[
|D^{2}u(0)|\leq C(n,\delta,\varepsilon).
\]

\end{lemma}

\begin{proof}
Set
\[
v=u+\alpha\frac{1}{\left|  \nabla u\left(  0\right)  \right|  }\left\langle
\nabla u\left(  0\right)  ,x\right\rangle \ \ \ \ \text{or\ \ \ \ }u+\alpha
x_{1}\ \ \text{if \ \ }\nabla u\left(  0\right)  =0,
\]
where $\alpha=\left(  \frac{1}{\sqrt{n-1}}-\delta\right)  /2.$ \ \ Now $v$
satisfies in $B_{1}$ the following
\[
D^{2}v=D^{2}u,\ \ \left|  \nabla v\left(  0\right)  \right|  \geq
\alpha,\ \ \text{and}\ \ \left|  \nabla v\right|  \leq\alpha+\delta<\frac
{1}{\sqrt{n-1}}.
\]

Set $b$ $=V^{\frac{\varepsilon}{n}},$ and consider the function
\[
w=\eta b=\left[  |\nabla v|^{2}-\left(  \alpha+\delta\right)  ^{2}%
|x|^{2}\right]  ^{+}b\geq0.
\]
A\ positive maximum for $w$ will be attained at a point $p$ on the interior,
since $w\left(  0\right)  >0$ and $w\left(  x\right)  $ vanishes on the
boundary $\partial B_{1}.$ \ At this point $p,$
\[
\nabla_{g}\left(  \eta b\right)  =0\ \ \ \ \ \text{or}\ \ \ \ \nabla_{g}%
\eta=-\frac{\eta}{b}\nabla_{g}b,
\]%
\begin{align*}
0  &  \geq\bigtriangleup_{g}\left(  \eta b\right)  =\eta\bigtriangleup
_{g}b+2\left\langle \nabla_{g}\eta,\nabla_{g}b\right\rangle _{g}%
+b\bigtriangleup_{g}\eta\\
&  =\eta\left(  \bigtriangleup_{g}b-2\frac{\left\vert \nabla_{g}b\right\vert
_{g}^{2}}{b}\right)  +b\bigtriangleup_{g}\eta\\
&  \geq b\bigtriangleup_{g}\eta,
\end{align*}
by the inequality (\ref{Jac1}) in Lemma 2.1. This last inequality implies a
bound on $|D^{2}v(p)|$ as the following. We have
\begin{align*}
0  &  \geq\bigtriangleup_{g}\eta=\bigtriangleup_{g}\left[  |\nabla
v|^{2}-\left(  \alpha+\delta\right)  ^{2}|x|^{2}\right] \\
&  =\sum_{i,j=1}^{n}g^{ij}\left[  2\sum_{k=1}^{n}\left(  v_{ki}v_{kj}%
+v_{k}\partial_{ij}v_{k}\right)  -\left(  \alpha+\delta\right)  ^{2}%
\partial_{ij}|x|^{2}\right] \\
&  =2\sum_{i=1}^{n}\frac{\lambda_{i}^{2}-\left(  \alpha+\delta\right)  ^{2}%
}{1+\lambda_{i}^{2}}\geq2\left[  \frac{\lambda_{1}^{2}-\left(  \alpha
+\delta\right)  ^{2}}{1+\lambda_{1}^{2}}-(n-1)\left(  \alpha+\delta\right)
^{2}\right]  ,
\end{align*}
using the minimal surface equation (\ref{Emin}) and assuming $\left\vert
\lambda_{1}\right\vert \geq\left\vert \lambda_{i}\right\vert $ for all $i.$ It
follows that
\[
1+\lambda_{1}^{2}\left(  p\right)  \leq\frac{1+\left(  \alpha+\delta\right)
^{2}}{1-(n-1)\left(  \alpha+\delta\right)  ^{2}}.
\]
We get
\[
\alpha^{2}b\left(  0\right)  \leq\left\vert \nabla v\left(  0\right)
\right\vert ^{2}b\left(  0\right)  \leq\eta\left(  p\right)  b\left(
p\right)  \leq\left(  \alpha+\delta\right)  ^{2}\left[  \frac{1+\left(
\alpha+\delta\right)  ^{2}}{1-(n-1)\left(  \alpha+\delta\right)  ^{2}}\right]
^{\frac{\varepsilon}{2}},
\]
then
\begin{equation}
1+\lambda_{i}^{2}\left(  0\right)  \leq\left(  1+\frac{\delta}{\alpha}\right)
^{\frac{4n}{\varepsilon}}\left[  \frac{1+\left(  \alpha+\delta\right)  ^{2}%
}{1-(n-1)\left(  \alpha+\delta\right)  ^{2}}\right]  ^{n}.
\label{2. bound at 0}%
\end{equation}

Therefore, we conclude the estimate $|D^{2}u(0)|\leq C(n,\delta,\varepsilon)$
in Lemma 2.2.
\end{proof}

\begin{lemma}
Let $u\in C^{\infty}\left(  \mathbb{R}^{n}\backslash\left\{  0\right\}
\right)  $ be a solution to the special Lagrangian equation (\ref{EsLag}) and
homogeneous of order $2$; that is, $u\left(  x\right)  =\left|  x\right|
^{2}u\left(  x/\left|  x\right|  \right)  .$ Suppose that the eigenvalues
$\lambda_{i}$ of the Hessian $D^{2}u\left(  x\right)  $ satisfy (\ref{Cond}).
Then $u$ must be quadratic.
\end{lemma}

\begin{proof}
Lemma 2.3 follows from Proposition 3.1; nonetheless we give a direct proof in
the following. Considering (\ref{Cond}), (\ref{Squeez}), and (\ref{Lowerbound}%
), we observe that the coefficients of $h_{ijk}^{2}$ in (\ref{Simons}) are
strictly positive. Accordingly,
\begin{equation}
\bigtriangleup_{g}\ln V\geq c\left(  \lambda\right)  \sum_{i,j,k=1}^{n}%
h_{ijk}^{2} \label{StrictSimons}%
\end{equation}
with $c\left(  \lambda\right)  >0.$

Since $u$ is homogeneous of order $2,$ the homogeneous order $0$ function $\ln
V$ attains its maximum along a ray. We infer from the strong maximum principle
that $\ln V\equiv const.$ It follows from (\ref{StrictSimons}) that
$D^{3}u\equiv0$. Therefore, $u$ must be quadratic, as claimed in Lemma 2.3.
\end{proof}

\begin{proof}
[Proof of Theorem 1.1]Now the Hessian bound is available by Lemma 2.2. We run
the \textquotedblleft routine\textquotedblright\ blow-down procedure
\textquotedblleft in detail\textquotedblright\ to finish the proof of Theorem
1.1, as in [Y2].

Step 1. From the assumption that $|\nabla u(x)|\leq\delta|x|$ for large $x, $
we have on the ball $B_{R}\left(  p\right)  $ with any fixed $p\in
\mathbb{R}^{n}$%
\[
|\nabla u(x)|\leq\delta\left(  \left\vert p\right\vert +R\right)  =\left(
\delta+\frac{\delta\left\vert p\right\vert }{R}\right)  R.
\]
A rescaled version of Lemma 2.2 with $R$ going to $\infty$ then leads to a
Hessian bound, $|D^{2}u(p)|\leq C(n,\delta,\varepsilon)\triangleq K,$ which
must hold at each point $p\in\mathbb{R}^{n}.$

Step 2. Repeating verbatim the argument in [Y1, p.263--264], we show that we
can find a tangent cone of the special Lagrangian graph $\left(  x,\nabla
u\left(  x\right)  \right)  $ at $\infty$ whose potential function is
$C^{1,1}, $ homogenous order $2,$ and still satisfies the
``convexity''\ condition (\ref{Cond}).

Without loss of generality, we assume $u\left(  0\right)  =0,$ $\triangledown
u\left(  0\right)  =0.$ We ``blow down''\ $u$ at $\infty.$

Set
\[
u_{k}\left(  x\right)  =\frac{u\left(  kx\right)  }{k^{2}},\;k=1,2,3,\cdots.
\]
We see that
\[
\left\|  u_{k}\right\|  _{C^{1,1}\left(  B_{R}\right)  }\leq C\left(
K,R\right)  ,
\]
so there exists a subsequence, still denoted by $\left\{  u_{k}\right\}  $ and
a function $u_{R}\in C^{1,1}\left(  B_{R}\right)  $ such that $u_{k}%
\rightarrow u_{R}$ in $C^{1,\alpha}\left(  B_{R}\right)  $ as $k\rightarrow
\infty,$ and $\left|  D^{2}u_{R}\right|  \leq K.$ By the fact that the family
of viscosity solution is closed under $C^{0}$ uniform limit, we know that
$u_{R}$ is also a viscosity solution of
\[
F\left(  D^{2}u\right)  =\sum_{i=1}^{n}\arctan\lambda_{i}=c\ \ \text{on}%
\ \ B_{R}.
\]
Applying the $W^{2,\delta}$ estimate (cf. [CC] Proposition 7.4) to the
difference $u_{k}-u_{R},$ we have
\[
\left\|  D^{2}u_{k}-D^{2}u_{R}\right\|  _{L^{\delta}\left(  B_{R/2}\right)
}\leq C\left(  K,R\right)  \left\|  u_{k}-u_{R}\right\|  _{L^{\infty}\left(
B_{R}\right)  }\rightarrow0\;\text{as }k\rightarrow\infty.
\]
Note that $\left|  D^{2}u_{k}\right|  ,\left|  D^{2}u_{R}\right|  \leq K,$ so
also
\[
\left\|  D^{2}u_{k}-D^{2}u_{R}\right\|  _{L^{n}\left(  B_{R/2}\right)
}\rightarrow0\;\text{as }k\rightarrow\infty.
\]
By a standard fact from real analysis, there exists another subsequence and
$C^{1,1}$ function on $B_{R}$, still denoted by $\left\{  u_{k}\right\}  $ and
$u_{R/2}$ such that $D^{2}u_{k}\rightarrow D^{2}u_{R/2}$ almost everywhere as
$k\rightarrow\infty.$ So $D^{2}u_{R}$ still satisfies (\ref{Cond}) almost
everywhere on $B_{R/2}.$

The diagonalizing process yields yet another subsequence, again denoted by
$\left\{  u_{k}\right\}  $ and $v\in C^{1,1}\left(  \mathbb{R}^{n}\right)  $
such that $u_{k}\rightarrow v$ in $W_{loc}^{2,n}\left(  \mathbb{R}^{n}\right)
$ as $k\rightarrow\infty;$ $v$ is a viscosity solution of (\ref{EsLag}) on
$\mathbb{R}^{n};\ \left|  D^{2}v\right|  \leq K;$ and $D^{2}v$ still satisfies
(\ref{Cond}) almost everywhere on $\mathbb{R}^{n}.$

The surfaces $\left(  x,\triangledown u_{k}\left(  x\right)  \right)  $ are
minimal in $\mathbb{R}^{n}\times\mathbb{R}^{n}$ and their potentials $u_{k}$
converge to $v$ in $W_{loc}^{2,n}\left(  \mathbb{R}^{n}\right)  ,$ so by the
monotonicity formula (cf. [S, p.84, Theorem 19.3] ), we conclude that
$M_{v}=\left(  x,\triangledown v\left(  x\right)  \right)  $ is a cone.

Step 3. We claim that $M_{v}$ is smooth away from the vertex. Suppose $M_{v}$
is singular at $P$ away from the vertex. We blow up $M_{v}$ at $P$ to get a
tangent cone, which is a lower dimensional special Lagrangian cone crossing a
line; repeat the procedure if the resulting cone is still singular away from
the vertex. Finally we get a special Lagrangian cone which is smooth away from
the vertex, and the bounded eigenvalues of the Hessian of the potential
function satisfies (\ref{Cond}), by a similar $W^{2,\delta}$ argument as in
Step 2. By Lemma 2.3, the cone is flat. This is a contradiction to Allard's
regularity result (cf. [S, Theorem 24.2]).

Applying Lemma 2.3 to $M_{v},$ we see that $M_{v}$ is flat.

Step 4. Now with the flatness of $M_{v}$, a final application of the
monotonicity formula yields that the original gradient graph $(x,\nabla u(x))
$ is also a plane (cf. [Y2, p.123]). Therefore, $u$ is a quadratic polynomial.
\end{proof}

\section{Appendix}

We include here a uniqueness result for global solutions to the special
Lagrangian equation (\ref{EsLag}) with bounded Hessian satisfying certain
``convexity''\ constraints (\ref{W-Cond}). The constraints are only needed for
$n\geq4.$

\begin{theorem}
Let $u$ be a smooth solution to the special Lagrangian equation (\ref{EsLag}).
Suppose that the eigenvalues $\lambda_{i}$ of the Hessian $D^{2}u\left(
x\right)  $ are bounded $\left|  \lambda_{i}\left(  x\right)  \right|  \leq K$
and satisfy
\begin{equation}
3+\lambda_{i}^{2}\left(  x\right)  +2\lambda_{i}\left(  x\right)  \lambda
_{j}\left(  x\right)  \geq0\quad\label{W-Cond}%
\end{equation}
for all $i,\ j,\ $and $x.$ Then $u$ must be a quadratic polynomial.
\end{theorem}

\begin{proof}
The proof is identical to the one of Theorem 1.1 with Lemma 2.3 replaced by
the following proposition.
\end{proof}

\begin{proposition}
Let $u\in C^{\infty}\left(  \mathbb{R}^{n}\backslash\left\{  0\right\}
\right)  $ be a solution to the special Lagrangian equation (\ref{EsLag}) and
homogeneous of order $2$, that is $u\left(  x\right)  =\left\vert x\right\vert
^{2}u\left(  x/\left\vert x\right\vert \right)  .$ Suppose that the
eigenvalues $\lambda_{i}$ of the Hessian $D^{2}u\left(  x\right)  $ satisfy
(\ref{W-Cond}) for all $i,\ j,\ $and $x\neq0.$ Then $u$ must be quadratic.

\begin{proof}
By (\ref{W-Cond}), we certainly have (\ref{Jac-Diag}) with $\varepsilon=0$ in
Lemma 2.1, that is
\begin{align}
\bigtriangleup_{g}\ln V  &  \geq\sum_{i=1}^{n}\left(  1+\lambda_{i}%
^{2}\right)  h_{iii}^{2}+\sum_{j\neq i}\left(  3+\lambda_{j}^{2}+2\lambda
_{i}\lambda_{j}\right)  h_{jji}^{2}\nonumber\\
&  =\sum_{i=1}^{n}\frac{1}{\left(  1+\lambda_{i}^{2}\right)  ^{2}}u_{iii}%
^{2}+\sum_{j\neq i}\frac{\left(  3+\lambda_{j}^{2}+2\lambda_{i}\lambda
_{j}\right)  }{\left(  1+\lambda_{j}^{2}\right)  ^{2}\left(  1+\lambda_{i}%
^{2}\right)  }u_{jji}^{2}\geq0. \label{JD2}%
\end{align}
Since $u$ is homogeneous of order $2$, the Hessian $D^{2}u\left(  x\right)  $
is homogeneous of order $0$, hence $\ln V$ must attain its maximum along a
ray. The strong maximum principle yields that $\ln V$ is constant, so in fact
\begin{equation}
0=\bigtriangleup_{g}\ln V. \label{volconst}%
\end{equation}
We claim now that
\begin{equation}
\bigtriangleup u\;=const \label{traceconst}%
\end{equation}
on $\mathbb{R}^{n}\backslash\left\{  0\right\}  .$ At any point $p$ compute
the derivative
\begin{equation}
\partial_{i}\;(\bigtriangleup u)=\sum_{j}u_{jji} \label{tobecontra}%
\end{equation}
for all $i.$ \ Still assuming that $D^{2}u$ is diagonalized at $p$, an
inspection of (\ref{JD2}), together with (\ref{volconst}) shows that for all
$j$ with $u_{jji}\neq0$,
\begin{equation}
3+\lambda_{j}^{2}+2\lambda_{i}\lambda_{j}\ =0.\quad\label{eq32}%
\end{equation}
From $3+\lambda_{i}^{2}+2\lambda_{i}\lambda_{j}\geq0,$ we see that
$\lambda_{i}^{2}\geq\lambda_{j}^{2}.$ Solving (\ref{eq32}) for $\lambda_{j}$
we get\
\[
\lambda_{j}=-\lambda_{i}-\sqrt{\lambda_{i}^{2}-3},\text{ if }\lambda
_{i}<0\text{ }%
\]%
\[
\lambda_{j}=-\lambda_{i}+\sqrt{\lambda_{i}^{2}-3},\text{ if }\lambda
_{i}>0\text{ .}%
\]
The minimal surface equation (\ref{Emin})\ at $p$ then reads
\[
0=\bigtriangleup_{g}u_{i}\overset{p}{=}\sum_{j}\frac{1}{1+\lambda_{j}^{2}%
}\,u_{jji}=\frac{1}{1+\left(  -\lambda_{i}\pm\sqrt{\lambda_{i}^{2}-3}\right)
^{2}}\sum_{j}\,u_{jji}.
\]
Hence $\partial_{i}(\bigtriangleup u)=0$ and $\bigtriangleup u$ is constant. \ \ \ 

Differentiating (\ref{traceconst}), we see that each $u_{ij}$ satisfies
\[
\bigtriangleup u_{ij}=0.
\]
Applying the strong maximum principle once again to each (homogeneous order
$0$) function $u_{ij}$, we have immediately
\[
u_{ij}=const;
\]
that is, $u$ is quadratic.
\end{proof}
\end{proposition}

\end{document}